\documentclass[reqno, 12pt ,a4letter]{amsart}
\usepackage{amsmath,amsxtra,amssymb,latexsym, amscd,amsthm}
\usepackage{graphicx,color}
\usepackage[utf8]{inputenc}
\usepackage[mathscr]{euscript}
\usepackage{mathrsfs}
\usepackage[english]{babel}

\newtheorem {theorem}{Theorem}[section]
\newtheorem {corollary}[theorem]{Corollary}
\newtheorem {lemma}[theorem]{Lemma}
\newtheorem {example}[theorem]{Example}

\newtheorem {remark}[theorem]{Remark}
\newtheorem {claim}{Claim}

\def\ees{{\accent"5E e}\kern-.385em\raise.2ex\hbox{\char'23}\kern-.08em}
\def\EES{{\accent"5E E}\kern-.5em\raise.8ex\hbox{\char'23 }}
\def\ow{o\kern-.42em\raise.82ex\hbox{
\vrule width .12em height .0ex depth .075ex \kern-0.16em \char'56}\kern-.07em}
\def\OW{O\kern-.460em\raise1.36ex\hbox{
\vrule width .13em height .0ex depth .075ex \kern-0.16em \char'56}\kern-.07em}

\title[Topological invariants of plane curve singularities]{Topological invariants of plane curve singularities: Polar quotients and \L ojasiewicz gradient exponents}

\author[HONG-DUC NGUYEN]{HONG-DUC NGUYEN$^\dag$}
\address{$^{\dag}$Hanoi Institute of Mathematics\newline \indent 18 Hoang Quoc Viet, Hanoi, Vietnam.} 
\email{nhduc82@gmail.com}
\address{$^{\dag}$Basque Center for Applied Mathematics, \newline \indent Alameda de Mazarredo 14, 48009 Bilbao, Bizkaia, Spain.} 
\email{hnguyen@bcamath.org}

\author[TI\EES N-S\OW N PH\d{A}M]{TI\EES N-S\OW N PH\d{A}M$^{\ddag}$}
\address{$^{\ddag}$Department of Mathematics, \newline \indent University of Dalat, \newline \indent 1 Phu Dong Thien Vuong, Dalat, Vietnam}
\email{sonpt@dlu.edu.vn}

\author[PHI-D\~{U}NG HO\`{A}NG]{PHI-D\~{U}NG HO\`{A}NG$^{\S}$}
\address{$^{\S}$Laboratory of Applied Mathematics and Computing - Faculty of Fundamental Sciences,
\newline \indent  Posts and Telecommunications Institute of Technology
\newline \indent 
Km10 Nguyen Trai Rd., Ha Dong District, Hanoi, Vietnam}
\email{dunghp@ptit.edu.vn}

\thanks{$^{\dag}$This author is supported by the National Foundation for Science and Technology Development (NAFOSTED), Grant number 101.04-2017.12, Vietnam, the ERCEA Consolidator Grant 615655 NMST and also by the Basque Government through the BERC 2014-2017 program and by Spanish Ministry of Economy and Competitiveness MINECO: BCAM Severo Ochoa excellence accreditation SEV-2013-0323.}

\thanks{$^{\ddag, \S}$These authors are supported by the National Foundation for Science and Technology Development (NAFOSTED), Grant number 101.04-2016.05, Vietnam.}

\keywords{Plane curve singularity, Polar curve, Polar quotient, \L ojasiewicz exponent, Newton polygon, Newton--Puiseux root, Topological invariant}

\subjclass{32S05, 58K65, 14H20}

\date{ \today}

\begin{document}
\maketitle

\begin{abstract}
In this paper, we study polar quotients and \L ojasiewicz exponents of plane curve singularities, which are {\em not necessarily reduced}. 
We first show that the polar quotients is a topological invariant. 
We next prove that the \L ojasiewicz gradient exponent can be computed in terms of the polar quotients, and so it is also a topological invariant.
As an application, we give effective estimates of the \L ojasiewicz exponents in the gradient and classical inequalities of polynomials in two (real or complex) variables.
\end{abstract}

\section{Introduction}

The polar quotients (called also polar invariants) of isolated hypersurface singularities were introduced by Teissier in the seventies of the last century to study equisingularity problems 
(\cite{Teissier1975, Teissier1976, Teissier1977}). They are, by definition, the quotients of  the contact orders between a hypersurface  and the branches of its generic polar curve. It is proved that the set of polar quotients is an analytic invariance, and in the case of reduced plane curve singularities, it is a topological invariant (see \cite{Teissier1977}). The Milnor number, the \L ojasiewicz gradient exponent and other numerical invariants can be computed in terms of  the polar quotients. Teissier's polar quotients can be easily adapted to non-isolated hypersurface singularities.  However, it seems to be more difficult to obtain similar results in the general case.

One of our main results is to show that, in the case of (not necessarily reduced) plane curve singularities, the polar quotients and the \L ojasiewicz gradient exponent are topological invariants. More precisely, we will show in Section~\ref{Section3} that the set of polar quotients of a plane curve singularity can be interpreted in terms of approximations of its Newton--Puiseux roots (Theorem~\ref{Theorem31}). Then, using a recent result due to Parusi\'nski \cite{Parusinski2008}, we obtain the topological invariance of the set of polar quotients (Theorem~\ref{Theorem34}).

Let $f\in \mathbb{K}\{z_1, \ldots, z_n\}$ with $\mathbb{K}=\mathbb{C}$ or $\mathbb{R}$ be a hypersurface singularity. Taking its value in a small neighbourhood, $f$ can be identified with an analytic function germ $f \colon (\mathbb{K}^n, 0) \to (\mathbb{K}, 0)$. It is well-known (see \cite{Lejeune1974, Lojasiewicz1965}) that there exist positive numbers $c, \ell$ and $\epsilon$ such that the following {\em \L ojasiewicz gradient inequality} holds
\begin{eqnarray*}
\|\nabla f(z)\| &\ge& c\, |f(z)|^\ell \quad \textrm{ for all } \quad \|z\| < \epsilon.
\end{eqnarray*}
The minimum of such exponent $\ell$ is called the {\em \L ojasiewicz gradient exponent} of $f$ (at the origin) and is denoted by $\mathscr{L}(f).$ The number $\mathscr{L}(f)$ is rational belonging to the interval $[0, 1)$ and above inequality holds with any exponent $\ell \ge \mathscr{L}(f)$ for some positive constants $c, \epsilon.$
It will be shown in Section~\ref{Section4} that, if $n=2$, i.e. $f$ is a plane curve singularity, then the  \L ojasiewicz gradient exponent $\mathscr{L}(f)$ is attained along the polar curve of $f$ and so it can be computed in terms of the polar quotients of $f$ (Theorems~\ref{Theorem41}, \ref{Theorem45}, and Corollary~\ref{Corollary43}). In particular, 
the \L ojasiewicz gradient exponent of complex plane curve singularities is a topological invariant (Corollary~\ref{Corollary44}). 

As an application, we give effective estimates for the \L ojasiewicz exponents in the gradient and classical inequalities of polynomials. Namely, if $f$ is a (real or complex) polynomial in two variables of degree $d,$ then (Theorem~\ref{Theorem49})
\begin{eqnarray*}
\mathscr{L}(f) & \le & 1 - \frac{1}{(d - 1)^2 + 1}.
\end{eqnarray*}
From this we derive the following effective version of the {\em classical {\L}ojasiewicz inequality} 
\begin{eqnarray*}
|f(z)| & \ge & c\, \mathrm{dist}(z, f^{-1}(0))^{(d - 1)^2 + 1} \quad \textrm{ for all } \quad \|z\| < \epsilon,
\end{eqnarray*}
where $\mathrm{dist}(z, f^{-1}(0))$ denotes the distance from $z$ to the set $f^{-1}(0)$ (Theorem~\ref{Theorem410}).  We refer the reader to the papers \cite{Acunto2005, Gwozdziewicz1999, Johnson2011, Kollar1999, Pham2012} for recent results concerning the estimation of the {\L}ojasiewicz  exponents for (real) polynomials in higher dimensions.

Our proofs are based on the notion of Newton polygon relative to an arc, which will be recalled in Section~\ref{Section2}.

\section{The Newton polygon relative to an arc} \label{Section2}

The technique of Newton polygons plays an important role in this paper. It is well-known that Newton transformations which arise in a natural way when applying the Newton algorithm provide a useful tool for calculating invariants of singularities. For a complete treatment we refer to \cite{ Brieskorn1986, Casas-Alvero2000, Walker1950, Wall2004}. In this section we recall the notion of Newton polygon relative to an arc due to Kuo and Parusi\'nski \cite{Kuo2000} (see also, \cite{HD08} and \cite{HD10}).

Let $f \colon (\mathbb{K}^2, 0) \to (\mathbb{K}, 0)$ denote an analytic function germ with Taylor expansion:
$$f(x, y) = f_m(x, y) + f_{m + 1}(x, y) +  \cdots,$$
where $f_k$ is a homogeneous polynomial of degree $k,$ and $f_m \not \equiv 0.$ We will assume that $f$ is {\em mini-regular in $x$ of order $m$} in the sense that $f_m(1, 0) \ne 0.$
(This can be achieved by a linear transformation $x' = x, y' = y + cx,$ $c$ a generic constant). Let $\phi$ be an analytic arc in $\mathbb{K}^2$, which is not tangent to the $x$-axis. Then it can be parametrized by 
$$x=c_1t^{n_1} + c_2t^{n_2}+ \cdots \in \mathbb{K}\{t\} \text{ and } y=t^N$$
and therefore can be identified with a {\em Puiseux series}
\begin{eqnarray*}
x = \phi(y) = c_1y^{n_1/N} + c_2y^{n_2/N}+ \cdots \in \mathbb{K}\{y^{1/N}\}
\end{eqnarray*}
with $N \le n_1 < n_2 < \cdots $ being positive integers. Let us apply the change of variables $X := x - \phi(y)$ and $Y := y$ to $f(x, y),$ yielding 
$$F(X, Y) := f(X + \phi(Y), Y) := \sum c_{ij}X^iY^{j/N}.$$
For each  $c_{ij} \ne 0,$ let us plot a dot at $(i, j/N),$ called a {\em Newton dot.} The set of Newton dots is called the {\em Newton diagram.} They generate a convex hull, whose boundary is called the {\em Newton polygon of $f$ relative to $\phi,$}  to be denoted by $\mathbb{P}(f, \phi).$ Note that this is the Newton polygon of $F$ in the usual sense. If $\phi$ is a {\em Newton-Puiseux root} of $f = 0$ (i.e. $f(\phi(y), y)=0$), then there are no Newton dots on $X = 0$, and vice versa. Assume that $\phi$ is not a Newton-Puiseux root of $f$, then the exponents of the series $f(\phi(y), y) = F(0, Y)$ correspond to the Newton dots on the line $X = 0.$ In particular, $\mathrm{ord} f(\phi(y), y) = h_0,$ where $(0, h_0)$ is the lowest Newton dot on $X = 0$. The Newton edges $E_s$ and their associated Newton angles $\theta_s$ are defined in an obvious way as illustrated in the following example.
\begin{example}{\rm
Take $f(x, y) := {x}^{3}-{y}^{4}+{y}^{5}$ and $\phi : x = y^{4/3}.$ 
We have 
$$F(X, Y) := f(X + \phi(Y), Y) = {X}^{3}+3\,{X}^{2}{Y}^{4/3}+3\,X{Y}^{8/3}+{Y}^{5}.$$ 
By definition, the Newton polygon of $f$ relative to $\phi$ has two compact edges $E_1, E_2$ with $\tan \theta_1 = 4/3, \tan \theta_2 = 7/3$ (see Figure~\ref{Figure1}).

\unitlength = .75cm
\begin{figure}
\begin{picture}(7, 8)(-2, -1)
\put (0, 0){\vector(0, 1){6}}
\put (0, 0){\vector(1, 0){6}}

\multiput(0, 1) (.5, 0){12}{\line(1, 0){.1}}
\multiput(0, 2) (.5, 0){12}{\line(1, 0){.1}}
\multiput(0, 3) (.5, 0){12}{\line(1, 0){.1}}
\multiput(0, 4) (.5, 0){12}{\line(1, 0){.1}}
\multiput(0, 5) (.5, 0){12}{\line(1, 0){.1}}

\multiput(1, 0) (0, .5){12}{\line(0, 1){.1}}
\multiput(2, 0) (0, .5){12}{\line(0, 1){.1}}
\multiput(3, 0) (0, .5){12}{\line(0, 1){.1}}
\multiput(4, 0) (0, .5){12}{\line(0, 1){.1}}
\multiput(5, 0) (0, .5){12}{\line(0, 1){.1}}

\multiput(1, 2.67) (-.45, 0){2}{\line(-1, 0){.2}}

\put(3, 0){\circle*{0.2}}
\put(2, 1.3){\circle*{0.2}}
\put(1, 2.67){\circle*{0.2}}
\put(0, 5){\circle*{0.2}}

\put(1.6, 2.2){$E_1$}
\put(.5, 4.2){$E_2$}

\put(2.2, .15){$\theta_1$}
\put(0.3, 2.85){$\theta_2$}

\thicklines
\put(3, 0){\line(-3, 4){2}}
\put(1, 2.67){\line(-2, 5){.91}}

\put(-.75, 4.85){$\ 5$}
\put(-.75, 2.5){$\frac{8}{3}$}
\put(-.75, 1.15){$\frac{4}{3}$}
\put(.85, -1){$1$}
\put(1.85, -1){$2$}
\put(2.85, -1){$3$}

\put(1, -.15){\line(0, 1){.3}}
\put(2, -.15){\line(0, 1){.3}}
\put(-.15, 2.67){\line(1, 0){.3}}
\put(-.15, 1.3){\line(1, 0){.3}}
\end{picture}
\caption{ \ } \label{Figure1}
\end{figure}
}\end{example}

Take any edge $E_s.$ The associated polynomial $\mathcal{E}_s(z)$ is defined to be $\mathcal{E}_s(z) := \mathcal{E}_s(z, 1),$ where
$$\mathcal{E}_s(X, Y) := \sum_{(i, j/N) \in E_s} c_{ij} X^i Y^{j/N}.$$ 
The {\em highest Newton edge}, denoted by $E_H,$ is the compact edge of  the polygon $\mathbb{P}(f, \phi)$ with a vertex being the lowest Newton dot on $X = 0.$ For instance, in the above example, $E_2$ is the highest Newton edge.

Next, we recall the notion of {\em sliding} (see \cite{Kuo2000}). Suppose that $\phi$ is not a root of $f  = 0.$  Consider the Newton polygon $\mathbb{P}(f, \phi).$ Take any nonzero root $c$ of $\mathcal{E}_H(z) = 0,$ the polynomial equation associated to the highest Newton edge $E_H.$ We call
$$\phi_1(y) : x = \phi(y) + cy^{\tan \theta_H}$$
a {\em sliding} of $\phi$ along $f,$ where $\theta_H$ is the angle associated to $E_H.$ A recursive sliding $\phi \to \phi_1 \to \phi_2 \to \cdots$ produces a limit, $\phi_\infty$, which is a root of $f = 0$. The $\phi_\infty$ will be called a {\em final result of sliding $\phi$ along $f$}. Note that,  the $\phi_\infty$ has the form
$$\phi_\infty\colon x = \phi(y) + c y^{\tan \theta_H}+\text{higher order terms},$$
due to the following lemma. 

\begin{lemma} \label{Lemma22}
Let $\phi$ be a Puiseux series, which is not root of $f = 0$ and let $\theta_H$ and $\mathcal{E}_H$ be the angle and polynomial associated to the highest Newton edge $E_H.$ Consider a series 
$$\psi : x = \phi(y) + c y^{\rho} + \textrm{ higher order terms,}$$
where $c \in \mathbb{K}$ and $\rho \in \mathbb{Q}, \rho > 0.$ Then the following statements hold:
\begin{itemize}
\item[(i)] If either $\tan \theta_{H}< \rho$ or $\tan \theta_{H}= \rho$ and $\mathcal{E}_H(c) \ne 0$ then 
$\mathbb{P}(f, \phi) = \mathbb{P}(f, \psi),$ and therefore 
$\mathrm{ord} f(\phi(y), y) = \mathrm{ord} f(\psi(y), y).$

\item[(ii)] If $\tan \theta_{H}= \rho$ and $\mathcal{E}_H(c)  = 0$ then $\mathrm{ord} f(\phi(y), y) < \mathrm{ord} f(\psi(y), y).$
\end{itemize}
\end{lemma}

\begin{proof}
This is well-known as the Newton-Puiseux algorithm of finding a Newton--Puiseux root of $f = 0$ (cf. \cite{Brieskorn1986, Casas-Alvero2000, Walker1950}). For a detailed proof, we refer to \cite{HD08}. In fact, the special case where $\psi : x = \phi(y) + c y^{\tan \theta_H}$ was proved in \cite[Lemma 2.1]{HD08}. The lemma is then deduced by applying the special case (infinitely) many times.
\end{proof}

\section{Polar quotients} \label{Section3}

Let $f \colon (\mathbb{C}^2, 0) \to (\mathbb{C}, 0)$ be an analytic function germ which is mini-regular in $x.$ After Teissier \cite{Teissier1977}, we define the set of {\em polar quotients} 
$$\mathcal{Q}(f) := \{\mathrm{ord} f(\gamma(y), y) \ | \ \gamma \in \Gamma(f)\}.$$
In this section we will show that the set of polar quotients is a topological invariant. We first give a formula interpreting polar quotients in terms of approximations of Newton--Puiseux roots of $f$. Let $\gamma(y) := \sum_{i} a_i y^{\alpha_i}$ be a Puiseux series. For each positive real number $\rho$, the series $\sum_{\alpha_i < \rho}a_{i}y^{\alpha_i} + gy^\rho,$ where $g$ is a generic constant, is called the {\em $\rho$-approximation} of $\gamma(y)$. For two distinct Puiseux series $\gamma_1(y)$ and $\gamma_2(y)$, the {\em approximation of $\gamma_1(y)$ and $\gamma_2(y)$} is defined to be the $\rho$-approximation series of $\gamma_1(y)$ (and hence of $\gamma_2(y)$), where $\rho := \mathrm{ord}\ (\gamma_1(y) - \gamma_2(y))$ is the {\em contact order} of $\gamma_1(y)$ and $\gamma_2(y)$. 

\begin{theorem}\label{Theorem31}
Let $f \colon (\mathbb{C}^2, 0) \to (\mathbb{C}, 0)$ be an analytic function germ which is mini-regular in $x$ of order $m$ and let $\xi_1, \ldots, \xi_r$ $(r \ge 2)$ be its distinct Newton--Puiseux roots. Then 
$$\mathcal{Q}(f) = \left\{\mathrm{ord}\, f(\xi_{i,j}(y),y) \mid 1 \le i < j  \le r \right\},$$ 
where $\xi_{i,j}$ denotes the approximation of $\xi_{i}$ and $\xi_{j}.$
\end{theorem}

\begin{proof}
Take any $\gamma \in \Gamma(f)$ and consider the Newton polygon $\mathbb{P}(f, \gamma)$ of $f$ relative to $\gamma.$ Note that $(m, 0)$ is a vertex of the Newton polygon and there is a dot on the line $X = 0$ but there are not dots on the line $X = 1.$ 
Let $E_H$ and $\mathcal{E}_H$ be the highest edge and the corresponding associated polynomial.
We have
$$\deg \mathcal{E}_H \geq 2, \quad \mathcal{E}_H(0) \ne 0, \quad \textrm{ and } \quad \frac{d}{dz } \mathcal{E}_H(0) = 0.$$
It implies that the equation $\mathcal{E}_H(z) = 0$ has at least two non-zero distinct roots, say $c_1, c_2.$ Let $\gamma_{i, \infty}, i = 1, 2$ be a final result of sliding of the arc $y \mapsto \gamma(y) + c_i y^{\tan \theta_H}$ along $f,$ where $\theta_H$ is the angle corresponding  to the highest edge $E_H.$ 
We have $\mathrm{ord} (\gamma_{1, \infty}(y) - \gamma_{2, \infty}(y)) = \tan \theta_H$ and
$$f(\gamma_{i, \infty}(y), y) \equiv 0 \quad \textrm{ for } \quad i = 1, 2.$$
Let $\gamma_{1, 2}$ be the approximation of $\gamma_{1, \infty}$ and $\gamma_{2, \infty}.$
It follows from Lemma \ref{Lemma22} that
$\mathrm{ord} f(\gamma(y), y)  = \mathrm{ord} f(\gamma_{1, 2}(y), y),$ and hence
$$\mathcal{Q}(f) \subset \left\{\mathrm{ord} f(\xi_{i,j}(y),y) \mid 1 \le i < j  \le r \right\}.$$ 

To show the inverse inclusion, we take any pair of distinct roots $\xi_1, \xi_2$ of $f$ and let  $\xi_{1,2}$ be the approximation of $\xi_{1}$ and $\xi_{2}$. Then we may write 
$$\xi_{1,2}(y)=\xi_{1}(y) +g y^{\rho}+\text{higher order terms}$$
where $g$ is generic, and $\rho$ denotes the contact order of $\xi_{1}$ and $\xi_{2}$. Write
$$f(X + \xi_1(Y), Y) = \sum c_{ij} X^i Y^{j/N}.$$

Let $\Delta$ be the (non empty) set of Newton dots in the Newton polygon $\mathbb{P}(f, \xi_1)$ of $f$ relative to $\xi_1,$
where the linear function $(i, j) \mapsto \rho i + j/N$ defined on $\mathbb{P}(f, \xi_1)$ takes its minimal value, say $h_0.$ We denote by $(i_1,j_1/N)$ the lowest point of $\Delta$, i.e. the point in $\Delta$ with maximal $i_1.$ Since $\xi_1$ is a root of $f = 0,$ there are no dots on the line $X = 0,$ and so $i_1 \ge 1.$

Let us denote $\phi(y):=\xi_{1}(y) +g y^{\rho}$ and
$$F(X, Y) :=  f(X + \phi(Y), Y) =\sum c_{ij} (X + g Y^\rho)^i Y^{j/N}= \sum a_{ij} X^i Y^{j/N}.$$
Note that, all the Newton dots of the Newton polygon $\mathbb{P}(f, \phi)$ of $f$ relative to $\phi$ must have the form $(k, \rho(i - k) + j/N)$ with $c_{i j}\neq 0$ and $k = 0, 1, \ldots, i$, and $(i_1,j_1/N)$ is a Newton dot of $\mathbb{P}(f, \phi)$ because $a_{i_1 j_1} = c_{i_1 j_1}\neq 0.$ Since $g$ is generic, the point $(0,h_0)$ is also a Newton dot of $\mathbb{P}(f, \phi)$ (in fact, we take $g$ satisfying $\sum_{(i, j/N)\in \Delta} c_{ij}g^i\neq 0$). Furthermore, all the Newton dots of $\mathbb{P}(f, \phi)$ lie on or above the line containing the two dots $(0,h_0)$ and $(i_1,j_1/N)$. This shows that the edge $E_H$ connecting $(0,h_0)$ and $(i_1,j_1/N)$ is the highest Newton edge of $\mathbb{P}(f, \phi).$ Let $\theta_H$ and $\mathcal{E}_H(z)$ be the angle and polynomial associated with the highest Newton edge $E_H.$ We have (see Figure~\ref{Figure2})
\begin{eqnarray*}
\tan \theta_H &=& \frac{h_0 - j_1/N}{i_1} \ = \ \frac{\rho i_1}{i_1} \ = \ \rho, \\
\mathcal{E}_H(z) &=& \sum_{(i, j/N)\in E_H} a_{i j}z^i \ = \ \sum_{(i, j/N) \in \Delta} c_{i j}(z + g)^i.
\end{eqnarray*}
By the definition of $\phi$, we may write 
\begin{eqnarray*}
\xi_1(y) &=& \phi(y) + a_1 y^\rho + \textrm{ higher order terms,} \\
\xi_2(y) &=&  \phi(y) + a_2 y^\rho + \textrm{ higher order terms}
\end{eqnarray*}
with $a_1\neq a_2.$ It follows from Lemma \ref{Lemma22} that $a_1$ and $a_2$ are roots of the polynomial $\mathcal{E}_H(z).$ In particular, we have $\deg \mathcal{E}_H(z) \geq 2.$

\unitlength = .75cm
\begin{figure}
\begin{picture}(7, 15)(-2, -1.5)
\put (0, 0){\vector(0, 1){13}}
\put (0, 0){\vector(1, 0){7}}

\put (2.35, 5){\vector(3, 1){2}}

\multiput(4.32, -1) (-.25, .75){5}{\line(-1, 3){.16}}
\multiput(2, 6) (-.25, .75){9}{\line(-1, 3){.16}}
\multiput(3, 3) (-.45, 0){7}{\line(-1, 0){.2}}

\thicklines
\put(5, 0){\line(-1, 1){1}}
\put(4, 1){\line(-1, 2){1}}
\put(3, 3){\line(-1, 3){1}}
\put(2, 6){\line(-1, 5){1}}

\put(5, 0){\circle*{0.2}}
\put(4, 1){\circle*{0.2}}
\put(3, 3){\circle*{0.2}}
\put(2, 6){\circle*{0.2}}
\put(0, 12){\circle*{0.2}}
\put(1, 11){\circle*{0.2}}

\put(-1.75, 0){$(0, 0)$}
\put(5, -1){$(m, 0)$}
\put(-1.75, 12){$(0, h_0)$}
\put(3.5, 2.85){$(i_1, \frac{j_1}{N})$}
\put(2.75, 4.5){$\Delta$}
\put(2.75, 7){$\mathbb{P}(f, \xi_1)$}
\put(2., 3.25){$\theta_H$}
\put(4.5, 5.5){$(\rho, 1)$}
\end{picture}
\caption{ \ } \label{Figure2}
\end{figure}

Consider the Newton diagram of $\frac{\partial f}{\partial x}$ relative to $\phi$ which can be easily found from $\mathbb{P}(f, \phi)$. Namely, move every Newton dot $(i, j/N)$ of $\mathbb{P}(f, \phi)$ to $(i - 1, j/N),$ if $i \ge 1,$ and delete all Newton dots $(0, j/N).$ This is simply because $\frac{\partial}{\partial X}(X^i Y^{j/N}) = i X^{i - 1} Y^{j/N}$. Since $g$ is generic, 
$\frac{d}{dz}\mathcal{E}_H(g) \ne 0$ and so, in the Newton diagram of $f$ to $\phi$ there exists a Newton dot on the line $X = 1.$
Therefore the highest Newton edge of  the Newton polygon $\mathbb{P}(\frac{\partial f}{\partial x}, \phi)$ has the vertices at $(0, h_0 - \tan \theta_H)$ and $(i_1- 1, j_1/N).$ The associated polynomial equation is $\frac{d}{dz} \mathcal{E}_H(z) = 0$. Since $\mathcal{E}_H(z)$ has two distinct roots $a_1,a_2$, it follows by simple calculation (see also the argument in the proof of Theorem \ref{Theorem41}), that there exists a nonzero number $c \in \mathbb{C}$ such that
$$\mathcal{E}_H(c) \ne 0 \quad \textrm{ and } \quad \frac{d}{dz}\mathcal{E}_H (c) = 0.$$
Let $\gamma_\infty$ be a final result of sliding of the arc $\phi$ along $\frac{\partial f}{ \partial x }$. It follows from Lemma \ref{Lemma22} that
$$\mathrm{ord} f(\gamma_\infty(y), y) = \mathrm{ord}f(\phi(y), y).$$ 
Furthermore, since $\mathrm{ord}\left(\xi_{1,2}(y)-\phi(y)\right)>\rho = \tan \theta_H$, applying again Lemma \ref{Lemma22} we get
$$\mathrm{ord} f(\xi_{1,2}(y), y) = \mathrm{ord}f(\phi(y), y).$$
Hence the inverse inclusion holds. The theorem is proved.
\end{proof}

By using the same argument as in the proof of Theorem~\ref{Theorem31}, we obtain

\begin{corollary}
Let $f \colon (\mathbb{C}^2, 0) \to (\mathbb{C}, 0)$ be an analytic function germ and let $\xi_1,\ldots,\xi_r$ $(r \ge 2)$ be its distinct Newton--Puiseux roots. Then for each pair of distinct roots $\xi_i,\xi_j$ there exists a curve $\gamma \in \Gamma(f)$ such that
$$\mathrm{ord} (\gamma-\xi_k)=\mathrm{ord} (\xi_{i,j}-\xi_k) \quad  \textrm{ for } \quad k = 1, \ldots, r,$$ 
where $\xi_{i,j}$ denotes the approximation of $\xi_{i}$ and $\xi_{j}.$
\end{corollary}

The above corollary is a sharper version of \cite[Lemma~3.3]{Kuo1977}. Indeed, by letting $k = i$ and $ k= j,$ we get 
$$\mathrm{ord} (\gamma-\xi_i)=\mathrm{ord} (\gamma-\xi_j)=\mathrm{ord} (\xi_{i}-\xi_j).$$ 

Recall that, two continuous function germs $f, g \colon (\mathbb{K}^2, 0) \to (\mathbb{K}, 0)$ are said to be {\em topologically right equivalent}, if there exists a germ of homeomorphisms $h \colon (\mathbb{K}^2, 0) \to (\mathbb{K}^2, 0)$ such that $f = g \circ h.$ In \cite{Kuo1977} Kuo and Lu introduced a tree model of an isolated singularity $f \colon (\mathbb{C}^2, 0) \to (\mathbb{C}, 0).$ This model allows one to visualise the Puiseux pairs of irreducible components of $f = 0$ and the contact orders between them. Kuo and Lu's model can be easily adapted to the nonisolated case by adding the multiplicities of components. 

We need the following result due to Parusi\'nski~\cite[Theorem 0.1 and Remark 0.4]{Parusinski2008}, where the last statement follows directly from the proof therein.

\begin{theorem}\label{Theorem33}
Let $f , g  \colon (\mathbb{C}^2, 0) \to (\mathbb{C}, 0)$ be (not necessarily reduced) analytic function germs. Then the following are equivalent
\begin{itemize}
\item[(i)] $f$ and $g$ are topologically right equivalent.
\item[(ii)] There is a one-to-one correspondence between the irreducible components of the zero sets $f^{-1}(0)$ and $g^{-1}(0)$  that preserves the multiplicities of these components, their Puiseux pairs, and the intersection multiplicities of any pairs of distinct components.
\item[(iii)] The tree models of $f$ and of $g$ coincide.
\item[(iv)] There is a one-to-one correspondence between the distinct Newton--Puiseux roots of $f = 0$ and $g = 0$ that preserves the multiplicities of these roots, and the contact orders of any pairs of distinct roots.
\end{itemize}
\end{theorem}

\begin{theorem}\label{Theorem34}
The set of polar quotients of (not necessarily reduced) complex analytic function germs in two variables is a topological invariant.
\end{theorem}

\begin{proof}
Let $f \colon (\mathbb{C}^2, 0) \to (\mathbb{C}, 0)$ be an analytic function germ which is mini-regular in $x$ and let $\xi_1,\ldots,\xi_r$ be its distinct Newton--Puiseux roots with multiplicities $m_1,\ldots,m_r$:
$$f(x, y) = u(x, y) \prod_{k = 1}^r (x - \xi_k(y))^{m_k},$$
where $u$ is a unit in $\mathbb{C}\{x,y\}.$ If $r = 1$ then the set of polar quotients of $f$ is empty and there is nothing to prove. So assume that $r \ge 2.$ Denote by $\xi_{i,j}$ the approximation of two distinct roots $\xi_i$ and $\xi_j.$ We have
$$\mathrm{ord}\ f(\xi_{i,j}(y),y)=\sum_{k=1}^{r}m_k\mathrm{ord} (\xi_{i,j}-\xi_k)=\sum_{k=1}^{r}m_k\min\{\mathrm{ord} (\xi_{i}-\xi_k),\mathrm{ord} (\xi_{j}-\xi_k)\}.$$
The theorem follows immediately from Theorems~\ref{Theorem31}~and~\ref{Theorem33}.
\end{proof}

\section{\L ojasiewicz exponents}\label{Section4}

Let $f \colon (\mathbb{K}^n, 0) \to (\mathbb{K}, 0)$ be an analytic function germ. Take any analytic arc $\phi$ parametrized by 
$$\phi(t) = \left(z_1(t), \ldots, z_n(t)\right),$$
where each $z_i(t)$ is a convergent power series, for $|t|$ small. If $f \circ \phi \not \equiv 0,$ then we can define a positive rational number $\ell(\phi)$ by
\begin{eqnarray*}
\| \nabla f (\phi(t)) \| & \simeq & |f(\phi(t))|^{\ell(\phi)},
\end{eqnarray*}
where $A \simeq B$ means that $A/B$ lies between two positive constants. By the Curve Selection Lemma (see \cite{Milnor1968}), it is not hard to show that the {\L}ojasiewicz gradient exponent of $f$ is given by
\begin{eqnarray}\label{Eqn1}
\mathscr{L}(f) &=& \sup_\phi \ell(\phi),
\end{eqnarray}
where the supremum is taken over all analytic curves passing through the origin, which are not contained in the zero locus of $f.$  
As a consequence, considering a generic linear curve, one can see that 
\begin{eqnarray} \label{Eqn3}
\mathscr{L}(f) &\ge& \frac{m - 1}{m},
\end{eqnarray}
where $m := \mathrm{ord}\, f$ stands for the multiplicity of $f$ at the origin.
Moreover, from \eqref{Eqn1} and the inequality $\mathscr{L}(f) < 1,$ it is not hard to see that for any unit $u$ in $\mathbb{K}\{z_1,\ldots, z_n\},$
\begin{eqnarray} \label{Eqn4}
\mathscr{L}(u \cdot f) &=& \mathscr{L}(f).
\end{eqnarray}

In the two next subsections we provide formulas computing the {\L}ojasiewicz gradient exponent of analytic function germs in two real and complex variables. 

\subsection{\L ojasiewicz gradient exponent of complex analytic function germs} \

Let $f \colon (\mathbb{C}^2, 0) \to (\mathbb{C}, 0)$ be an analytic function germ. Assume that $f$ is mini-regular in $x$ of order $m := \mathrm{ord} f$. Recall that, the loci defined by $\frac{\partial f}{ \partial x } = 0$ is called a {\em polar curve.}  Following \cite{Walker1950}, a Newton--Puiseux root of $\frac{\partial f}{ \partial x } = 0$ is called a {\em branch} of the polar curve or simply a {\em polar branch.} We denote by $\Gamma(f)$ the set of polar branches which are not roots of $f = 0.$

\begin{theorem}\label{Theorem41}
With the above notations, the \L ojasiewicz gradient exponent of $f$ is given by
$$\mathscr{L}(f) = 
\begin{cases}
\frac{m - 1}{m} & \textrm{ if } \Gamma(f) = \emptyset, \\
\max\left\{ \ell(\gamma)\mid \gamma\in \Gamma(f)\right\} & \textrm{ otherwise.}
\end{cases}$$
\end{theorem}

\begin{proof}
We first consider the case $\Gamma(f) \neq \emptyset.$ By the Weierstrass preparation theorem (see, for instance, \cite{Brieskorn1986, Greuel2006}), there exist $u \in \mathbb{C}\{x, y\}$ 
and $a_i \in \mathbb{C}\{y\}$ such that
$$g(x, y) := u(x, y) \cdot f(x, y) = x^m + a_1(y) x^{m - 1} + \cdots + a_m(y)$$
with $u(0, 0) \ne 0$ and $a_i(0) = 0.$ Due to the division theorem (see, for instance, \cite{Brieskorn1986, Greuel2006}), there exist $\phi \in \mathbb{C}\{y\},$ with $\phi(0) = 0,$ and a polynomial $h\in \mathbb{C}\{y\}[x]$ of degree at most $m-2$ such that 
$$m g(x,y)=(x - \phi(y)) \frac{\partial g}{ \partial x}(x, y)+h(x,y),$$
or, equivalently, 
$$m u(x,y) \cdot f(x,y)=(x - \phi(y)) \left(\frac{\partial u}{ \partial x}(x, y)f(x,y)+u(x,y)\frac{\partial f}{ \partial x}(x, y)\right)+h(x,y).$$
Since $\Gamma(f) = \emptyset,$ it follows that all the $m - 1$ roots of $\frac{\partial f}{ \partial x} = 0$ are also roots (counted with multiplicity) of $h = 0,$ and hence $h \equiv 0$ because $\deg h\leq m - 2.$ Then the differential equation 
$$\frac{\frac{\partial g}{ \partial x}(x, y)}{g(x,y)}=\frac{m}{x - \phi(y)}$$ 
implies that $g$ has the form $c(x - \phi(y))^m$ for some $c \neq 0.$ Consequently, one has
\begin{eqnarray*}
\|\nabla g(x, y) \| &\ge& \left|\frac{\partial g}{ \partial x}(x, y) \right| \ = \ m\, c^{\frac{1}{m}}\, |g(x, y)|^{\frac{m - 1}{m}} \quad \textrm{ for all $(x,y)$ near $(0, 0)$}.
\end{eqnarray*}
This implies that $\mathscr{L}(g) \leq \frac{m-1}{m}$ and that the equality because of \eqref{Eqn3}. Therefore, from \eqref{Eqn4}, we get
$$\mathscr{L}(f) \ = \ \mathscr{L}(g) \ = \ \frac{m - 1}{m}.$$

We now consider the case $\Gamma(f) \neq \emptyset.$ We will prove the inequality 
$$\mathscr{L}(f)\leq \max\left\{ \ell(\gamma)\mid \gamma\in \Gamma(f)\right\}.$$
By \eqref{Eqn1}, this is equivalent to showing that 
$$\ell(\phi)\leq \max\left\{ \ell(\gamma)\mid \gamma\in \Gamma(f)\right\}$$ 
for all analytic curves $\phi$ passing through the origin and not contained in the zero locus of $f$. To this end, we make the following observation. 

\begin{claim}\label{Claim1}
The inequality  
$$\ell(\gamma) \ge \ \frac{m - 1}{m}$$
holds  for all $\gamma \in \Gamma(f)$.
\end{claim}

\begin{proof}
Let $x = \gamma(y)$ be a Newton--Puiseux root of $\frac{\partial f}{ \partial x } = 0$ but not of $f = 0.$
Since $f$ is mini-regular in $x$ of order $m,$ $\frac{\partial f}{\partial x}$ is mini-regular in $x$ of order $m - 1.$ This implies that $\mathrm{ord}\, \gamma(y)\geq 1$ and so  $\mathrm{ord}\, f \left(\gamma(y),y\right) \geq m.$ Note that
\begin{eqnarray*}
\frac{d f (\gamma(y),y )}{ d y }
 = \frac{\partial f}{ \partial y }(\gamma(y),y),
\end{eqnarray*}
and hence $\mathrm{ord}\, f \left(\gamma(y),y\right) = \mathrm{ord}\, \frac{\partial f}{ \partial y } \left(\gamma(y),y\right) +1.$ It yields that
\begin{eqnarray*}
\ell(\gamma ) 
\ = \  \frac{\mathrm{ord}\, \frac{\partial f}{ \partial y }\left(\gamma(y),y\right)}{\mathrm{ord}\, f \left(\gamma(y),y\right)} \ \ge \ 1 - \frac{1}{m}.
\end{eqnarray*}
% Then it is easy to check that 
% \begin{eqnarray*}
% \mathrm{ord} (f \circ \gamma) \geq m \quad \textrm{ and  } \quad \mathrm{ord} (f  \circ \gamma) = \mathrm{ord} (\frac{\partial f}{ \partial y } \circ \gamma) + 1.
% \end{eqnarray*}
% Therefore,
% \begin{eqnarray*}
% \ell(\gamma ) 
% \ = \  \frac{\mathrm{ord} (\frac{\partial f}{ \partial y }\circ \gamma )}{\mathrm{ord} (f \circ \gamma )} = 1 - \frac{1}{\mathrm{ord} (f \circ \gamma )} 
% \ \ge \ 1 - \frac{1}{m}.
% \end{eqnarray*}
\end{proof}

Take any analytic arc $\phi$ which is not root of $f = 0.$ It is easy to see that if $\phi$ is tangent to the $x$-axis, then $\ell(\phi)\leq \frac{m - 1}{m}$. We can therefore ignore these arcs. Then the arc $\phi$ may be parametrized by a Puiseux series $x = \phi(y)$ with $\mathrm{ord}\, \phi(y) \geq 1.$ 

Assume that $\frac{\partial f}{ \partial x}(\phi(y), y) \not \equiv 0.$ In the Newton polygon $\mathbb{P}(f, \phi)$ of $f$ relative to $\phi,$ let $(0, h_0)$ and $(1, h_1)$ be the lowest Newton dots on $X = 0$ and $X = 1,$ respectively. Then $\ell(\phi)$ can be computed as follows.
\begin{claim}\label{Claim2} 
We have
$$\ell(\phi) = \min \Big \{\frac{h_0 - 1}{h_0}, \frac{h_1}{h_0} \Big \}.$$
\end{claim}

\begin{proof}
Let
\begin{eqnarray*}
F(X, Y) &:=& f(X + \phi(Y), Y) \\
&=& unit \cdot Y^{h_0} + unit \cdot Y^{h_1}X + \textrm{ terms divisible by } X^2.
\end{eqnarray*}
By the Chain Rule, 
\begin{eqnarray*}
\frac{\partial F}{ \partial X } &=& \frac{\partial f}{ \partial x } \quad \textrm{ and } \quad 
\frac{\partial F}{ \partial Y } \ = \  \phi'(Y) \frac{\partial f}{ \partial x } + \frac{\partial f}{ \partial y }.
\end{eqnarray*}
Since $\mathrm{ord}\ \phi(Y)\geq 1$, it follows that
\begin{eqnarray*}
\Big |\frac{\partial F}{ \partial X } \Big| + \Big|\frac{\partial F}{ \partial Y } \Big| &\simeq& \Big| \frac{\partial f}{ \partial x } \Big| + \Big |\frac{\partial f}{ \partial y } \Big|.
\end{eqnarray*}
Along $X = 0,$ we have
\begin{eqnarray*}
|F| \ \simeq \ |Y|^{h_0}, \quad  \Big |\frac{\partial F}{ \partial X } \Big| \ \simeq \ |Y|^{h_1}, \quad 
\Big |\frac{\partial F}{ \partial Y } \Big| &\simeq& |Y|^{h_0 - 1},
\end{eqnarray*}
whence the result.
\end{proof}
\begin{claim}\label{Claim3} 
Let $\gamma$ denote a final result of sliding of $\phi$ along $\frac{\partial f}{ \partial x }.$ Then
$$\ell(\phi) \le \ell(\gamma).$$
\end{claim}
\begin{proof}
In fact, consider the Newton polygon $\mathbb{P}(f, \phi).$ Let $E_H$ and $\theta_H$ be the highest edge and the corresponding angle. 
Note that $(0, h_0)$ is a vertex of $E_H$ and $(m, 0)$ is a vertex of the polygon. There are two cases to be consider (see Figure~\ref{Figure3}).
\setlength{\unitlength}{0.27cm}
\begin{figure}
\begin{center}
\begin{picture}(20, 25)(0, -5) \label{Fig2}
\linethickness{0.05mm}

\put (0, 0){\vector(0, 1){19}}
\put (0, 0){\vector(1, 0){22}}

\multiput(0, 0)(2, 0){10}{\line(0, 1){.2}}
\multiput(0, 2)(2, 0){10}{\line(0, 1){.2}}
\multiput(0, 4)(2, 0){10}{\line(0, 1){.2}}
\multiput(0, 6)(2, 0){10}{\line(0, 1){.2}}
\multiput(0, 8)(2, 0){10}{\line(0, 1){.2}}
\multiput(0, 10)(2, 0){10}{\line(0, 1){.2}}
\multiput(0, 12)(2, 0){10}{\line(0, 1){.2}}
\multiput(0, 14)(2, 0){10}{\line(0, 1){.2}}
\multiput(0, 16)(2, 0){10}{\line(0, 1){.2}}

\multiput(0, 0)(0, 2){9}{\line(1,0){.2}}
\multiput(2, 0)(0, 2){9}{\line(1,0){.2}}
\multiput(4, 0)(0, 2){9}{\line(1,0){.2}}
\multiput(6, 0)(0, 2){9}{\line(1,0){.2}}
\multiput(8, 0)(0, 2){9}{\line(1,0){.2}}
\multiput(10, 0)(0, 2){9}{\line(1,0){.2}}
\multiput(12, 0)(0, 2){9}{\line(1,0){.2}}
\multiput(14, 0)(0, 2){9}{\line(1,0){.2}}
\multiput(16, 0)(0, 2){9}{\line(1,0){.2}}
\multiput(18, 0)(0, 2){9}{\line(1,0){.2}}

\thicklines
\put(0,16){\line(1,-2){6}}
\put(6,4){\line(3, -2){6}}
\put(0,16){\circle*{0.5}}
\put(6,4){\circle*{0.5}}
\put(2,12){\circle*{0.5}}
\put(12,0){\circle*{0.5}}
\put(-1.4,-2){$0$}
\put(-5.,16){$(0, h_0)$}
\put(-5.,12){$(1, h_1)$}
\put(1.5,-2){$1$}
\put(10,-2){$(m, 0)$}
\put(0,-5){Case 1: $(1, h_1) \in E_H$}
\end{picture}\qquad \qquad \qquad
\begin{picture}(20, 25)(0, -5) \label{Fig2}
\linethickness{0.05mm}

\put (0, 0){\vector(0, 1){19}}
\put (0, 0){\vector(1, 0){22}}

\multiput(0, 0)(2, 0){10}{\line(0, 1){.2}}
\multiput(0, 2)(2, 0){10}{\line(0, 1){.2}}
\multiput(0, 4)(2, 0){10}{\line(0, 1){.2}}
\multiput(0, 6)(2, 0){10}{\line(0, 1){.2}}
\multiput(0, 8)(2, 0){10}{\line(0, 1){.2}}
\multiput(0, 10)(2, 0){10}{\line(0, 1){.2}}
\multiput(0, 12)(2, 0){10}{\line(0, 1){.2}}
\multiput(0, 14)(2, 0){10}{\line(0, 1){.2}}
\multiput(0, 16)(2, 0){10}{\line(0, 1){.2}}

\multiput(0, 0)(0, 2){9}{\line(1,0){.2}}
\multiput(2, 0)(0, 2){9}{\line(1,0){.2}}
\multiput(4, 0)(0, 2){9}{\line(1,0){.2}}
\multiput(6, 0)(0, 2){9}{\line(1,0){.2}}
\multiput(8, 0)(0, 2){9}{\line(1,0){.2}}
\multiput(10, 0)(0, 2){9}{\line(1,0){.2}}
\multiput(12, 0)(0, 2){9}{\line(1,0){.2}}
\multiput(14, 0)(0, 2){9}{\line(1,0){.2}}
\multiput(16, 0)(0, 2){9}{\line(1,0){.2}}
\multiput(18, 0)(0, 2){9}{\line(1,0){.2}}

\thicklines
\put(0,16){\line(1,-3){4}}
\put(4,4){\line(2,-1){4}}
\put(8,2){\line(3,-1){6}}
\put(0,16){\circle*{0.5}}
\put(4,4){\circle*{0.5}}
\put(8,2){\circle*{0.5}}
\put(2,12){\circle*{0.5}}
\put(14,0){\circle*{0.5}}
\put(-1.4,-2){$0$}
\put(-5,16){$(0, h_0)$}
\put(-5,12){$(1, h_1)$}
\put(1.5,-2){$1$}
\put(12,-2){$(m, 0)$}
\put(0,-5){Case 2: $(1, h_1) \not \in E_H$}
\end{picture}
\end{center}
\caption{ \ } \label{Figure3}
\end{figure}

\subsubsection*{Case 1: $(1, h_1) \in E_H$}
We have 
\begin{eqnarray*}
\frac{h_0 - h_1}{1} &=& \tan \theta_H \ \ge \ 1.
\end{eqnarray*}
Hence, $h_0 - 1 \ge h_1.$ By Claim~\ref{Claim2}, we get
\begin{eqnarray*}
\ell(\phi) 
&=& \min \Big \{\frac{h_0 - 1}{h_0}, \frac{h_1}{h_0} \Big \} \ = \  \frac{h_1}{h_0} \ \le \ \frac{m - 1}{m},
\end{eqnarray*}
where the last inequality follows from the assumption that $f$ is mini-regular in $x$ of order $m := \mathrm{ord} f.$

\subsubsection*{Case 2: $(1, h_1) \not \in E_H$}
In this case, $\theta_H <\theta_{H'}$ , where $\theta_{H'}$ denotes the angle corresponding to the highest edge of the Newton polygon of $\frac{\partial f}{ \partial x }$ relative to $\phi$. Since $\gamma$ is a final result of sliding of $\phi$ along $\frac{\partial f}{ \partial x }$, it has the form
$$\gamma(y) = \phi(y) + c y^{\tan \theta_{H'}}+\text{higher order terms},$$
for some nonzero constant $c \in \mathbb{C}.$ Applying Lemma~\ref{Lemma22} one has
$$\mathrm{ord} f(\gamma(y),y)=\mathrm{ord} f(\phi(y),y)=h_0.$$
It hence follows from Claim~\ref{Claim2} that 
\begin{eqnarray*}
\ell(\phi) 
&=& \min \Big \{\frac{h_0 - 1}{h_0}, \frac{h_1}{h_0} \Big \} \ \le \ \frac{h_0 - 1}{h_0} \ = \ \frac{\mathrm{ord} f(\gamma(y),y) - 1}{\mathrm{ord} f(\gamma(y), y)} \ = \ \ell(\gamma).
\end{eqnarray*}

Summing up in both cases we have
\begin{eqnarray*}
\ell(\phi)  &\le& \max \Big \{\frac{m - 1}{m}, \ell(\gamma) \Big \} \ \le \ \ell(\gamma),
\end{eqnarray*}
where the second inequality follows from Claim~\ref{Claim1}. 
\end{proof}

Applying the above claims we obtain that  
$$\mathscr{L}(f) \leq \max\left\{ \ell(\gamma)\mid \gamma\in \Gamma(f)\right\},$$
and hence the equality according to \eqref{Eqn1}. This completes the proof of Theorem~\ref{Theorem41}.
\end{proof}
\begin{example}{\rm
Take $f(x, y) =  1/6\,{x}^{6}+1/4\,{x}^{4}{y}^{4}-1/5\,{x}^{5}y-1/3\,{x}^{3}{y}^{5} \in \mathbb{C}\{x, y\}.$ We have $f$ is mini-regular in $x$ of order $m = 6$ and 
$\frac{\partial f}{ \partial x } = x^2(x - y)(x^2 + y^4).$ By definition, then $\Gamma(f)$ consists three polar branches
$$\gamma_{1} : x = y, \quad \gamma_{2} : x = \sqrt{-1} y^2, \quad  \textrm{ and } \quad \gamma_{3} :  x = - \sqrt{-1} y^2.$$
A simple calculation shows that $\ell(\gamma_{1} ) = \frac{5}{6}$ and $\ell(\gamma_{2}) = \ell(\gamma_{3}) = \frac{10}{11}.$ By Theorem~\ref{Theorem45}, 
$$\mathscr{L}(f) = \max \left \{\frac{5}{6}, \frac{10}{11} \right\} = \frac{10}{11}.$$
}\end{example}

The following result is a direct consequence of Theorems~\ref{Theorem31}~and~\ref{Theorem41}. It gives us an alternative formula to compute the \L ojasiewicz gradient exponent of $f$ in terms of its Newton--Puiseux roots.

\begin{corollary}\label{Corollary43}
Let $f \colon (\mathbb{C}^2, 0) \to (\mathbb{C}, 0)$ be an analytic function germ and let $\xi_1,\ldots,\xi_r$ be its distinct Newton--Puiseux roots. Then
$$\mathscr{L}(f) =
\begin{cases}
\frac{m-1}{m} & \textrm{ if } \ r = 1, \\
\max\left\{1 - \frac{1}{\mathrm{ord}\, f(\xi_{i,j}(y),y)} \mid 1\leq i<j\leq r\right\} & \textrm{ otherwise,}
\end{cases}$$
where $\xi_{i,j}$ denotes the approximation of $\xi_i$ and $\xi_j$.
\end{corollary}

\begin{corollary}\label{Corollary44}
The \L ojasiewicz gradient exponent of complex analytic function germs in two variables is a topological invariant.
\end{corollary}
\begin{proof}
This follows immediately from Corollary~\ref{Corollary43}, Theorems~\ref{Theorem31}~and~\ref{Theorem34}.
\end{proof}

\subsection{\L ojasiewicz gradient exponent of real analytic function germs} \

For an real analytic function germ $f \colon (\mathbb{R}^2, 0) \to (\mathbb{R}, 0)$  we have the following version of Theorem~\ref{Theorem41}. Let $x = \gamma(y)$ be a Newton--Puiseux root of $\frac{\partial f}{\partial x} = 0$ in the ring $\mathbb{C}\{x, y\}$: 
\begin{eqnarray*}
\gamma(y) = a_1 y^{n_1/N} + a_2y^{n_2/N}+ \cdots + a_{s - 1} y^{n_{s - 1}/N} + c_s y^{n_s/N} +  \cdots,
\end{eqnarray*}
where $a_i \in \mathbb{R},$ $c_s$ is the first non-real coefficient, if there is one. Let us replace $c_s$ by a generic real number $g,$ and call
\begin{eqnarray*}
\gamma_{\mathbb{R}}(y) := a_1 y^{n_1/N} + a_2y^{n_2/N}+ \cdots + a_{s - 1} y^{n_{s - 1}/N} + g y^{n_s/N},
\end{eqnarray*}
a {\em real polar branch.} In case $s = +\infty,$ let $\gamma_{\mathbb{R}} = \gamma.$ We denote by $\Gamma(f)$ the set of real polar branches of $f$ which are not Newton--Puiseux roots of $f = 0.$ Let
$$\mathscr{L}_{+}(f) := \max \left\{\frac{m-1}{m},\ell(\gamma_{\mathbb{R}}) \mid \gamma_{\mathbb{R}} \in \Gamma(f) \right\}.$$
We also put $\mathscr{L}_{-}(f) := \mathscr{L}_{+}(\bar{f}),$ where $\bar{f}$ denotes the germ defined by $\bar{f}(x, y) := f(x, -y).$

\begin{theorem}\label{Theorem45}
With the above notations, the \L ojasiewicz gradient exponent of $f$ is given by
$$\mathscr{L}(f) = \max \left\{\mathscr{L}_{+}(f),  \mathscr{L}_{-}(f) \right\}.$$
\end{theorem}

\begin{proof}
Let $\phi$ be a real curve parametrized by either $(x = x(t), y = t)$ or $(x = x(t), y = -t),$ where $x(t)$ is an element in $\mathbb{R}\{t^{1/N}\}$ for some positive integer number $N.$ 
Also assume that $\phi$ is not root of $f = 0.$ 

We first consider the case $\phi$ has the form $(x(t), t).$ Let us denote by $\gamma$ a final result of sliding of $\phi$ along $\frac{\partial f}{\partial x}$ and by $\gamma_{\mathbb{R}}$ its real approximation. 

\begin{claim}\label{Claim4} 
We have
$$\ell(\phi) \le \max \left \{\frac{m - 1}{m},\ell(\gamma_{\mathbb{R}}) \right\},$$
and therefore $\ell(\phi) \le \mathscr{L}_{+}(f)$.
\end{claim}

\begin{proof}
If $\phi$ is tangent to the $x$-axis, then $\ell(\phi)\leq \frac{m - 1}{m},$ and there is nothing to prove. 
So assume that the arc $\phi$ is parametrized by a Puiseux series $x = \phi(y)$ with $\mathrm{ord}\, \phi(y) \geq 1.$ 

Clearly, we may assume that $\phi$ is not root of $\frac{\partial f}{ \partial x} = 0.$ In the Newton polygon $\mathbb{P}(f, \phi)$ of $f$ relative to $\phi,$ let $(0, h_0)$ and $(1, h_1)$ be the lowest Newton dots on $X = 0$ and $X = 1,$ respectively. By the same argument as in Claim~\ref{Claim2}, the quantity $\ell(\phi)$ can be read off from the Newton polygon $\mathbb{P}(f,\phi)$ as
$$\ell(\phi) = \min \Big \{\frac{h_0 - 1}{h_0}, \frac{h_1}{h_0} \Big \}.$$
We can see moreover that, if $(1,h_1)\in E_H$, then $\ell(\phi)\leq \frac{m - 1}{m}$. It hence suffices to prove the claim for the case that $(1,h_1)\not\in E_H$. In this case, $\theta_H <\theta_{H'},$ where $\theta_{H'}$ denotes the angle corresponding to the highest Newton  edge of the Newton polygon of $\frac{\partial f}{ \partial x }$ relative to $\phi$. Since $\gamma$ is a final result of sliding of $\phi$ along $\frac{\partial f}{ \partial x }$, it has the form
$$\gamma(y)  = \phi(y) + c y^{\tan \theta_{H'}}+\text{higher order terms}$$ 
for some non-zero number $c \in\mathbb{C}.$ By definition, the series $\gamma_{\mathbb{R}}$ also has the form 
$$\gamma_{\mathbb{R}}(y) = \phi(y) + g y^{\tan \theta_{H'}}+\text{higher order terms}$$
with $g \in \mathbb{R}$ being generic if $c \not\in \mathbb{R}$ and $g = c$ otherwise.  Applying Lemma~\ref{Lemma22} for both $f$ and $\frac{\partial f}{\partial x}$, we obtain 
$$h_0 \ = \ \mathrm{ord} f(\phi(y),y) \ = \ \mathrm{ord} f(\gamma_{\mathbb{R}}(y),y)$$
and 
$$h_1 \ = \ \mathrm{ord}\  \frac{\partial f}{\partial x}(\phi(y),y) \ \le \ \mathrm{ord}\ \frac{\partial f}{\partial x} (\gamma_{\mathbb{R}}(y),y) \ =: \ h'_1.$$
It follows that
\begin{eqnarray*}
\ell(\phi) \ = \ \min \Big \{\frac{h_0 - 1}{h_0}, \frac{h_1}{h_0}\Big \} & \leq & \min \Big \{\frac{h_0 - 1}{h_0}, \frac{h'_1}{h_0}\Big \} \ = \ \ell(\gamma_{\mathbb{R}}).
\end{eqnarray*}
The claim is proved.
\end{proof}

We now consider the case $\phi$ has the form $(x(t), -t)$. Then
$$\|\nabla \bar{f}(x(t), t)\|  = \|\nabla {f}(x(t), -t)\| \simeq |{f}(x(t), -t)|^{\ell(\phi)} = |\bar{f}(x(t), t)|^{\ell(\phi)}.$$
Applying Claim \ref{Claim4} for $\bar{f}$ we get $\ell(\phi) \leq \mathscr{L}_{+}(\bar{f})=\mathscr{L}_{-}(f).$

Summing up in both cases we have
\begin{eqnarray*}
\ell(\phi)  &\le& \max \left \{\mathscr{L}_{+}(f), \mathscr{L}_{-}(f) \right\}.
\end{eqnarray*}
Since $\phi$ is arbitrary, we get easily from \eqref{Eqn1} and \eqref{Eqn3} that 
\begin{eqnarray*}
\mathscr{L}(f) & = & \max \left \{\mathscr{L}_{+}(f), \mathscr{L}_{-}(f) \right\}.
\end{eqnarray*}
The proof of theorem is completed.
\end{proof}

\begin{example}{\rm
Let $f(x,y) = x^3 + 3xy^3 \in \mathbb{R}\{x, y\}$. Then $f$ is mini-regular in $x$ of order $m =3$ and 
$\frac{\partial f}{ \partial x } = 3(x^2 + y^3).$ By definition, then $\Gamma(f)$ consists one real polar branch $\gamma : x = g y^{3/2} $
for some generic number $g.$ A simple calculation shows that $\ell(\gamma ) = \frac{2}{3}$ and that $\mathscr{L}_{+}(f)  = \frac{2}{3}$. It can be computed similarly that $\mathscr{L}_{-}(f)  = \frac{7}{9}$. Hence, $\mathscr{L}(f) = \frac{7}{9}$ by Theorem~\ref{Theorem45}.
}\end{example}

\begin{remark}{\rm
It is noting that the \L ojasiewicz gradient exponent of real analytic function germs is not topological invariant. Indeed, in some neighbourhood of the origin in $\mathbb{R}^2,$ consider the functions $f(x, y) := x^2 - y^3$ and $g(x, y) := x^2 - y^5$. It is obvious that they are topologically right equivalent. On the other hand, one can easily see that $\mathscr{L}(f) = 2/3 \ne 4/5 = \mathscr{L}(g).$
}\end{remark}

The following result was observed by Haraux \cite[Theorem~2.1]{Haraux2005} in the real case.
\begin{corollary}
Let $f \colon \mathbb{K}^2 \to \mathbb{K}$ be a homogeneous polynomial of degree $d.$ Then
$$\mathscr{L}(f) =  1 - \dfrac{1}{d}.$$
\end{corollary}
\begin{proof}
We first consider the case where $f$ is a complex homogeneous polynomial of degree $d.$  Since an invertible linear transformation does not change homogeneity of  $f,$ we may assume that $f$ is mini-regular in $x.$ We have $\frac{\partial f}{\partial x}$ is homogeneous polynomial of degree $d - 1.$ Hence, each root $x = \gamma(y)$ of $\frac{\partial f}{\partial x} = 0$ has the form $x = ay$ for some $a \in \mathbb{C}.$ Clearly, if $f(ay, y) \not \equiv 0,$ then $f(ay, y) = b y^d$ for some $b \ne 0,$ and so $\mathrm{ord} f(\gamma(y), y) = d = \mathrm{ord} f.$ Therefore, by Theorems~\ref{Theorem41},
$\mathscr{L}(f) =  1 - \frac{1}{d}.$

We now assume that $f$ is a real homogeneous polynomial of degree $d$ and consider its complexification $f_{\mathbb{C}}$. By definition, we have
\begin{eqnarray*}
\mathscr{L}(f) & \le & \mathscr{L}(f_{\mathbb{C}}) \ = \ 1 - \frac{1}{d}.
\end{eqnarray*}
On the other hand, the inequality $1 - \frac{1}{d} \le \mathscr{L}(f)$ holds. Therefore,
$\mathscr{L}(f) =  1 - \frac{1}{d}.$
\end{proof}

\subsection{Effective estimates for \L ojasiewicz exponents} \

In this subsection we give bounds for \L ojasiewicz exponents of polynomial functions in two variables. 
The bounds depend only on the degree of the polynomial and are simple to state.

\begin{theorem}[see also {\cite[Main Theorem]{Acunto2005}}] \label{Theorem49}
Let $f \colon \mathbb{K}^2 \to \mathbb{K}$ be a polynomial of degree $d$ with $f(0) = 0.$ Then
$$\mathscr{L}(f) \le 1 - \frac{1}{(d-1)^2 + 1}. $$
\end{theorem}

Before proving the corollary we recall the notion of intersection multiplicity of two plane curve germs (see, for example, \cite{Greuel2006}). Let $f \in \mathbb{C}\{x, y\}$ be irreducible. Then the {\em intersection multiplicity} of any $g \in \mathbb{C}\{x, y\}$ with $f$ is given by
$$i(f, g) := \mathrm{ord}\, g(x(t),y(t)),$$
where $t \mapsto (x(t), y(t))$ is a parametrization for the curve germ defined by $f.$ Here by a {\em parametrization} of the curve germ $f = 0,$ we mean an analytic map germ
$$\phi \colon (\mathbb{C}, 0) \rightarrow (\mathbb{C}^2, 0), \quad t \mapsto (x(t), y(t)),$$
with $f \circ \phi \equiv 0$ and satisfying the following {\em universal factorization property}:  each analytic map germ $\psi 
\colon (\mathbb{C}, 0) \rightarrow (\mathbb{C}^2, 0)$ with $f \circ \psi \equiv 0,$ there exists a unique analytic map germ 
$\psi' \colon (\mathbb{C}, 0) \rightarrow (\mathbb{C}, 0)$  such that $\psi = \phi \circ \psi'.$ In general, let $f \in \mathbb{C}\{x, y\}$ a convergent power series and let $f = f_1^{\alpha_1} \cdots f_r^{\alpha_r}$  be a factorization of $f$ in the ring $\mathbb{C}\{x, y\}$ with $f_i$ being irreducible and pairwise co-prime.
Then the intersection multiplicity of $g$ with $f$ is defined to be the sum
$$i(f_1^{\alpha_1} \cdots f_r^{\alpha_r}, g) := \alpha_1 i(f_1, g) + \cdots + \alpha_r i(f_r, g).$$

\begin{proof}[Proof of Theorem~\ref{Theorem49}]
By definition, if $f$ is a real polynomial, then $\mathscr{L}(f) \le \mathscr{L}(f_{\mathbb{C}}),$ where $f_{\mathbb{C}}$ is the complexification of $f.$ 
Hence, it suffices to consider the complex case.

Without loss of generality we may assume that $f$ is mini-regular in $x$ of order $m \le d.$
It follows from Theorem \ref{Theorem41} that, if $\Gamma(f) = \emptyset$ then
$$\mathscr{L}(f) = 1-\frac{1}{m}\leq 1-\frac{1}{(d-1)^2+1}.$$ 
We now assume that $\Gamma(f) \ne \emptyset.$ Take a polar branch $\gamma$ in $\Gamma(f)$, along which the \L ojasiewicz gradient exponent is attained:
$$\mathscr{L}(f) = \ell(\gamma) = 1 - \frac{1}{\mathrm{ord} \ f(\gamma(y),y)}.$$
Let $g$ be the irreducible factor of $\frac{\partial f}{ \partial x }$ in $\mathbb{C}\{x,y\}$ having $\gamma$ as a Newton--Puiseux root. Then $t \mapsto (\gamma(t^{N}), t^{N})$ is a parametrization of the curve germ $g = 0,$ where $N$ denotes the order of $g.$ Note that $i(\frac{\partial f}{\partial y}, g)$ is finite because $f \circ \gamma \not \equiv 0.$
We have
\begin{eqnarray*}
i \left (\frac{\partial f}{\partial y}, g \right) \ = \ \mathrm{ord}\, \frac{\partial f}{\partial y}(\gamma(t^{N}),t^{N}) &=& N \cdot\mathrm{ord}\, \frac{\partial f}{\partial y}(\gamma(y), y) \\
&\ge& \mathrm{ord} \, \frac{\partial f}{\partial y}(\gamma(y),y) = \mathrm{ord}\, f(\gamma(y),y)-1.
\end{eqnarray*}
Let $h \in \mathbb{cC}[x, y]$ be the irreducible component of the polynomial $\frac{\partial f}{ \partial x }$ which is, in $\mathbb{C}\{x,y\}$, divisible by $g.$ Note that $h$ does not divide $\frac{\partial f}{\partial y},$ since $i(\frac{\partial f}{\partial y}, g)$ is finite. It follows from Bezout's theorem (\cite{Brieskorn1986}, p.232) that 
$$i\left (\frac{\partial f}{\partial y}, h \right) \leq (d - 1) \cdot \deg h \leq (d-1)^2.$$
Since $i(\frac{\partial f}{\partial y}, g) \le i(\frac{\partial f}{\partial y}, h),$ therefore
$$\mathscr{L}(f)=1 - \frac{1}{\mathrm{ord} \ f(\gamma(y), y)}  \le  1 - \frac{1}{i(\frac{\partial f}{\partial y}, g) + 1}\leq 1 - \frac{1}{i(\frac{\partial f}{\partial y}, h) + 1} \leq  1-\frac{1}{(d - 1)^2 + 1}.$$
The corollary is proved.
\end{proof}

Let $f \colon \mathbb{K}^2 \to \mathbb{K}$ be a polynomial function of degree $d$ with $f(0) = 0.$ Set
\begin{eqnarray*}
\widetilde{\mathscr{L}}(f) &:=& \inf\{\ell \ | \ \exists c > 0, \exists \epsilon >0, |f(x, y)| \ge c\, \mathrm{dist}((x, y), f^{-1}(0))^\ell, \forall \|(x, y)\| < \epsilon\},
\end{eqnarray*}
where $\mathrm{dist}((x, y), f^{-1}(0))$ denotes the distance from $(x, y)$ to the set $f^{-1}(0)$ (see \cite{Lojasiewicz1959, Lojasiewicz1965}). It is well-known (see \cite{Bochnak1975, Kuo1974}) that the \L ojasiewicz exponent $\widetilde{\mathscr{L}}(f)$ is a rational number and it is attained along an analytic curve.

In case $\mathbb{K} = \mathbb{C},$ Risler and Trotman showed in \cite[Theorem~1]{Risler1997} that
\begin{eqnarray*}
\widetilde{\mathscr{L}}(f) &=& \mathrm{ord}\, f \ \le \ d.
\end{eqnarray*}
In case $\mathbb{K} = \mathbb{R},$ a formula for computing $\widetilde{\mathscr{L}}(f)$ was given by Kuo in \cite{Kuo1974}. 
Furthermore, we have the following result (see also \cite{Acunto2005, Gwozdziewicz1999, Johnson2011, Kollar1999, Kurdyka2014, Pham2012}).

\begin{theorem}\label{Theorem410}
Let $f \colon \mathbb{R}^2 \to \mathbb{R}$ be a real polynomial of degree $d$ with $f(0) = 0.$ Then
$$\widetilde{\mathscr{L}}(f) \le \frac{1}{1 - \mathscr{L}(f)}.$$
In particular, we have
$$\widetilde{\mathscr{L}}(f) \le (d - 1)^2 + 1.$$
\end{theorem}
\begin{proof}
The first inequality is an immediate consequence of the proof of Theorem~2.2 in \cite{Pham2012} (see also \cite{Kurdyka2014}). The second one can be deduced from Theorem~\ref{Theorem49}.
\end{proof}

\begin{remark}{\rm
In view of \cite[Example~1]{Kollar1999} (see also \cite{Gwozdziewicz1999, Johnson2011}), the estimate $\widetilde{\mathscr{L}}(f) \le (d - 1)^2 + 1$ is close to being optimal.
}\end{remark}

\subsection*{Acknowledgment.} 
A part of this work was done while the first author and the second author were visiting at Vietnam Institute for Advanced Study in Mathematics (VIASM) in the spring of 2016. These authors would like to thank the Institute for hospitality and support.

\bibliographystyle{abbrv}
\bibliography{ZDSB}

\begin{thebibliography}{10}

\bibitem{Bochnak1975}
J.~Bochnak and J.~J. Risler.
\newblock Sur les exposants de {{\L}}ojasiewicz.
\newblock {\em Comment. Math. Helv.}, 50(4):493--507, 1975.

\bibitem{Brieskorn1986}
E.~Brieskorn and H.~Kn\H{o}rrer.
\newblock {\em Plane algebraic curves}.
\newblock Birkh\H{a}user Verlag, Basel, 1986.

\bibitem{Casas-Alvero2000}
E.~Casas-Alvero.
\newblock {\em Singularities of plane curves}.
\newblock 276. Cambridge University Press, Cambridge, 2000.
\newblock London Mathematical Society. Lecture Note Series.

\bibitem{Acunto2005}
D.~D'Acunto and K.~Kurdyka.
\newblock Explicit bounds for the {{\L}}ojasiewicz exponent in the gradient
  inequality for polynomials.
\newblock {\em Ann. Polon. Math.}, 87:51--61, 2005.

\bibitem{Greuel2006}
G.-M. Greuel, C.~Lossen, and E.~Shustin.
\newblock {\em Introduction to singularities and deformations}.
\newblock Springer-Verlag, 2006.
\newblock Math. Monographs.

\bibitem{Gwozdziewicz1999}
J.~Gwo\'zdziewicz.
\newblock The {{\L}}ojasiewicz exponent of an analytic function at an isolated
  zero.
\newblock {\em Comment. Math. Helv.}, 74(3):364--375, 1999.

\bibitem{HD08}
H.~V. H\`a and H.~D. Nguyen.
\newblock On the {{\L}}ojasiewicz exponent near the fibre of polynomial
  mappings.
\newblock {\em Ann. Polon. Math.}, 94(1):43--52, 2008.

\bibitem{HD10}
H.~V. H\`a and H.~D. Nguyen.
\newblock {{\L}}ojasiewicz inequality at infinity for polynomials in two real
  variables.
\newblock {\em Math. Z.}, 266(2):243--264, 2010.

\bibitem{Haraux2005}
A.~Haraux.
\newblock Positively homogeneous functions and the {{\L}}ojasiewicz gradient
  inequality.
\newblock {\em Ann. Polon. Math.}, 87:165--174, 2005.

\bibitem{Johnson2011}
J.~M. Johnson and J.~Koll\'ar.
\newblock How small can a polynomial be near infinity?
\newblock {\em Amer. Math. Monthly}, 118(1):22--40, 2011.

\bibitem{Kollar1999}
J.~Koll\'ar.
\newblock An effective {{\L}}ojasiewicz inequality for real polynomials.
\newblock {\em Period. Math. Hungar.}, 38(3):213--221, 1999.

\bibitem{Kuo1974}
T.~C. Kuo.
\newblock Computation of {{\L}}ojasiewicz exponent of $f(x, y)$.
\newblock {\em Comment. Math. Helv.}, 49:201--213, 1974.

\bibitem{Kuo1977}
T.~C. Kuo and Y.~C. Lu.
\newblock On analytic function germs of two complex variables.
\newblock {\em Topology}, 16(4):299--310, 1977.

\bibitem{Kuo2000}
T.~C. Kuo and A.~Parusi\'nski.
\newblock Newton polygon relative to an arc.
\newblock In {\em Real and Complex Singularities}, volume 412, pages 76--93.
  Chapman \& Hall Res. Notes Math., 2000.

\bibitem{Kurdyka2014}
K.~Kurdyka and S.~Spodzieja.
\newblock Separation of real algebraic sets and the {{\L}}ojasiewicz exponent.
\newblock {\em Proc. Amer. Math. Soc.}, 142(9):3089--3102, 2014.

\bibitem{Lejeune1974}
M.~Lejeune-Jalabert and B.~Teissier.
\newblock Séminaire 1973--1974.
\newblock Centre de Mathématiques de l'Ecole Polytechnique.
\newblock A paraître.

\bibitem{Lojasiewicz1959}
S.~{{\L}}ojasiewicz.
\newblock Sur le probl\`eme de la division.
\newblock {\em Studia Math.}, (18):87--136, 1959.

\bibitem{Lojasiewicz1965}
S.~{{\L}}ojasiewicz.
\newblock Ensembles semi-analytiques.
\newblock {\em Publ. Math. I.H.E.S.}, 1965.

\bibitem{Milnor1968}
J.~Milnor.
\newblock {\em Singular points of complex hypersurfaces}.
\newblock Number~61. Princeton University Press, New Jersey, 1968.
\newblock Annals of Mathematics Studies.

\bibitem{Parusinski2008}
A.~Parusi\'nski.
\newblock A criterion for topological equivalence of two variable complex
  analytic function germs.
\newblock {\em Proc. Japan Acad. Ser. A Math. Sci.}, 84(8):147--150, 2008.

\bibitem{Pham2012}
T.~S. Ph\d{a}m.
\newblock An explicit bound for the {{\L}}ojasiewicz exponent of real
  polynomials.
\newblock {\em Kodai Math. J.}, 35(2):311--319, 2012.

\bibitem{Risler1997}
J.-J. Risler and D.~Trotman.
\newblock Bi-{{L}}ipschitz invariance of the multiplicity.
\newblock {\em Bull. London. Math. Soc.}, 29(2):200--204, 1997.

\bibitem{Teissier1975}
B.~Teissier.
\newblock Introduction to equisingularity problems.
\newblock In {\em Proc. Sympos. Pure Math.}, volume~29, pages 593--632,
  Humboldt State Univ., Arcata, Calif., 1975.
\newblock Amer. Math. Soc., Providence, R.I.

\bibitem{Teissier1976}
B.~Teissier.
\newblock The hunting of invariants in the geometry of discriminants.
\newblock In {\em Real and complex singularities}, pages 565--678, Nordic
  Summer School/NAVF Symposium in Mathematics. Oslo, 1976.

\bibitem{Teissier1977}
B.~Teissier.
\newblock Vari\'et\'es polaires. {I}.{ I}nvariants polaires des singularit\'es
  d'hypersurfaces.
\newblock {\em Invent. Math.}, 40(3):267--292, 1977.

\bibitem{Walker1950}
R.~J. Walker.
\newblock {\em Algebraic curves}.
\newblock Princeton University Press, 1950.

\bibitem{Wall2004}
C.~T.~C. Wall.
\newblock {\em Singular points of plane curves}, volume~63.
\newblock Cambridge University Press, Cambridge, 2004.
\newblock London Mathematical Society Student Texts.

\end{thebibliography}

\end{document}